\documentclass[11pt]{amsart}

\usepackage{amsmath,amssymb,amsthm}
\usepackage{geometry}
\usepackage{graphicx}
\usepackage{mathtools}
\usepackage{stmaryrd}
\usepackage{enumitem}
\newcommand{\jump}[1]{\llbracket #1 \rrbracket}

\theoremstyle{plain}
\newtheorem{theorem}{Theorem}[section]
\newtheorem{lemma}[theorem]{Lemma}

\theoremstyle{definition}

\theoremstyle{remark}

\title[]{Interface Reduction for Elliptic Interface Problems with Conservative Flux Reconstruction}

\author{C. Attanayake and  So-Hsiang Chou}

\address{
  Department of Mathematics,
Miami University,
Middletown, OH 45042, USA
}
\email{attanac@muohio.edu}
\address{Department of Mathematics and Statistics,
Bowling Green State University,
Bowling Green, OH 43403, USA}

\email{chou@bgsu.edu}

\date{\today}

\begin{document}
\maketitle
\begin{abstract}
We propose a low-dimensional interface reduction method for elliptic interface
problems based on conservative flux reconstruction. The approach combines a fitted
$P_1$ finite element discretization with a flux recovery procedure following
\cite{ChouTang2000}, yielding locally conservative fluxes that satisfy interface
conditions to machine precision.

A central result shows that the error of the reduced solution is controlled entirely
by the approximation error of the interface data. Numerical experiments for both
continuous and discontinuous interface conditions confirm that once the interface
data is accurately represented, the full solution is recovered to roundoff accuracy.

These results indicate that the essential complexity of elliptic interface problems
is concentrated on the interface.
\end{abstract}

\section{Introduction}

Elliptic interface problems arise in many applications involving composite media and multiphysics coupling. A key difficulty is that the solution is smooth within each subdomain while exhibiting sharp variation across the interface. In such problems, both the solution values and the normal fluxes across the interface play a fundamental role in determining the global behavior.

In this work, we propose an interface reduction method in which the global
solution is parametrized by a small number of interface degrees of freedom.
The approach is based on two key components: lifting techniques for handling
solution jumps and a flux reconstruction that produces locally conservative
fluxes with accuracy comparable to the primary variable.

The flux reconstruction is constructed using a recovery procedure
\cite{ChouTang2000} that enforces edge-based flux constraints. Unlike
standard postprocessing techniques, the recovered flux here plays a
structural role: it provides an approximation of the normal flux that is
consistent with the interface conditions and is of comparable accuracy
to the primary variable. This is essential in the reduced formulation,
where interface quantities involve both solution values and normal fluxes.

The central idea is to approximate the interface data in a reduced space while leaving the bulk discretization unchanged. This leads to a formulation in which the full solution is reconstructed from a small number of interface variables. When the interface data is represented exactly in the reduced space, the reduced solution coincides with the full discrete solution, demonstrating that elliptic interface problems admit an accurate low-dimensional description governed by interface quantities.

Elliptic interface problems have been widely studied using a variety of numerical approaches, including immersed interface methods, unfitted finite element methods, and domain decomposition techniques; see, for example, \cite{babuska1997,burman2015cutfem,leveque2007,lin2003,quarteroni1999}. Related developments in finite element analysis and discrete functional frameworks can be found in, e.g., \cite{ern2004}.

Postprocessing techniques for constructing locally conservative fluxes from continuous finite element solutions have also been investigated in various settings. In particular, \cite{ChouTang2000} introduced recovery procedures that enforce local conservation while maintaining consistency with the underlying finite element approximation. Other approaches to locally conservative flux reconstruction have been developed in different numerical contexts; see, for example, \cite{Deng2021SISC,DengGinting2018,ErnNicaiseVohralik2007}.

In contrast to classical domain decomposition or Schur complement approaches, the present method reduces the problem directly at the level of interface data without constructing global reduced spaces. This leads to a formulation in which the global solution is determined by interface variables.
In this setting, the flux reconstruction enables accurate evaluation of interface quantities
involving both solution values and normal fluxes.

These observations suggest that the essential complexity of elliptic interface problems is concentrated on the interface itself, motivating the reduction strategy developed in this work.

\section{Model Problems}

Let $\Omega \subset \mathbb{R}^2$ be a bounded domain divided by an interface $\Gamma$
into two subdomains $\Omega^\pm$. We consider a standard elliptic transmission problem
with piecewise constant, strictly positive diffusion coefficients $\beta^\pm$.

Find $u^\pm$ such that
\begin{align}
-\nabla \cdot (\beta \nabla u) = f \quad \text{in } \Omega^\pm,
\end{align}
subject to one of the following interface conditions:

(i) (flux-jump condition)
\begin{align}
\jump{u} = 0, \quad \jump{\beta \partial_n u} = g,
\end{align}

(ii) (solution-jump condition)
\begin{align}
\jump{u} = g, \quad \jump{\beta \partial_n u} = 0.
\end{align}

Here $\jump{u}$ and $\jump{\beta \partial_n u}$ denote the jumps of the solution and the normal
flux across the interface $\Gamma$, respectively.

For the bulk discretization, we employ a conforming $P_1$ finite element method.
Let $u_h$ denote the corresponding finite element approximation of $u$.
The flux $\mathbf{q} = -\beta \nabla u$ is approximated by a reconstructed flux $\mathbf{q}_h$,
which will be defined in the next section.
\section{Mixed Formulation and Flux Reconstruction}

The purpose of the flux reconstruction is to construct a locally conservative flux that enforces the interface conditions while maintaining accuracy comparable to the primary variable. This is essential for the interface reduction framework, where both solution values and fluxes enter the reduced model. The accuracy of the reduced model depends critically on the evaluation of normal fluxes, and a lower-order flux approximation would degrade the interface equations.

Let
\[
\mathbf{q} = -\beta \nabla u.
\]
For each element $K$, let $\beta_K$ denote the constant value of $\beta$ on $K$.
Given a conforming $P_1$ approximation $u_h$, we reconstruct a flux $\mathbf{q}_h$ elementwise by
\begin{equation}\label{eq:qh}
\mathbf{q}_h|_K=
-\beta_K \nabla u_h|_K
+
f_K \mathbf{P}_K
+
\mathbf{C}_K,
\qquad
\mathbf{P}_K(x)=\tfrac12(x-x_K),
\end{equation}
where $f_K = \frac{1}{|K|}\int_K f$, $x_K$ is the barycenter of $K$, and $\mathbf{C}_K \in \mathbb{R}^2$
is a constant vector.

Let $r_{K,e}$ denote the normal flux mismatch on the edge $e$, defined as
the difference between the target normal flux and the raw flux
$(-\beta_K \nabla u_h + f_K P_K)\cdot n_{K,e}$.
The correction is determined by enforcing edge flux constraints in a least-squares sense:
\begin{equation}\label{eq:CK}
\mathbf{C}_K=\arg\min_{C \in \mathbb{R}^2}
\sum_{e \subset \partial K}
\left(
\mathbf{C} \cdot \mathbf{n}_{K,e} - r_{K,e}
\right)^2,
\end{equation}
where $\mathbf{n}_{K,e}$ is the outward normal. A detailed construction and its relation to the classical flux recovery procedure of \cite{ChouTang2000} are given in Appendix~A.

For each element $K$, the correction $\mathbf{C}_K$ is the unique solution of \eqref{eq:CK}.
Equivalently, $\mathbf{C}_K$ satisfies the normal equations
\[
(N_K^T N_K) \mathbf{C}_K = N_K^T r_K,
\]
where $N_K$ is the matrix whose rows are the edge normals $\mathbf{n}_{K,e}$. Moreover, the correction is stable:
\[
|\mathbf{C}_K| \le C |r_K|_{\ell^2},
\]
and the reconstructed flux satisfies
\[
\|\mathbf{q}_h\|_{H(\mathrm{div}, \Omega^- \cup \Omega^+)}
\le C \bigl( \|\nabla u_h\|_{L^2(\Omega)} + \|f\|_{L^2(\Omega)} \bigr).
\]

Since the flux is discontinuous across the interface, the natural space is
\[
H(\mathrm{div}, \Omega^-) \oplus H(\mathrm{div}, \Omega^+)
= \{ \mathbf{q} : \mathbf{q}|_{\Omega^\pm} \in H(\mathrm{div}, \Omega^\pm) \}.
\]
We define the broken $H(\mathrm{div})$ norm by
\[
\|\mathbf{q}\|^2_{H(\mathrm{div},\Omega^- \cup \Omega^+)}
= \|\mathbf{q}\|^2_{L^2(\Omega)}+
\|\mathrm{div}\, \mathbf{q}\|^2_{L^2(\Omega^-)}+
\|\mathrm{div}\, \mathbf{q}\|^2_{L^2(\Omega^+)}.
\]
No continuity is imposed across the interface $\Gamma$.

Under standard regularity assumptions, the reconstructed flux satisfies the estimate
\[
\|\mathbf{q} - \mathbf{q}_h\|_{H(\mathrm{div}, \Omega^- \cup \Omega^+)}
\le C h \left( \|u\|_{H^2(\Omega^-)} + \|u\|_{H^2(\Omega^+)} \right).
\]
Thus, the flux reconstruction achieves the same order of accuracy as the primary variable.

The correction $C_K$ represents a minimal modification of the raw flux $-\beta \nabla u_h$
in the broken $H(\mathrm{div})$ setting that enforces edge flux constraints. The least-squares
construction ensures that the reconstructed flux is stable in the broken $H(\mathrm{div})$ norm
and has the same order of accuracy as the primary variable. This is crucial for the interface
reduction framework, where interface quantities involve both solution values and normal fluxes.

\section{Auxiliary Lifting  $E(\phi)$ for Solution-Jump Data}

To implement the interface reduction framework, we require a constructive operator that maps interface data to a function in the bulk. Given $\phi$ defined on the interface $\Gamma$, we construct a lifting
\[
w = E(\phi)
\]
such that
\begin{align}
\jump{w}_\Gamma = \phi, \qquad w|_{\partial\Omega} = 0.
\end{align}

This operator allows us to realize the reduced interface data $g_m\approx g$ in the discrete problem.

\subsection{1D Construction}

In one dimension, let $\Gamma = {\alpha}$. The lifting $w = E(\phi)$ can be constructed locally using the two elements sharing $\alpha$. Define $w$ as a piecewise linear function supported on these two elements such that:
\begin{enumerate}[label=(\alph*)]
\item $w(\alpha^+) - w(\alpha^-) = \phi$,
\item $w = 0$ at the endpoints of the two elements,
\item $w$ vanishes outside this local patch.
\end{enumerate}

This construction is local, minimal, and easy to implement.

\subsection{2D Fitted Mesh Construction}

For a fitted mesh, the interface $\Gamma$ aligns with element edges. Let $e$ be an interface edge shared by elements $K^+$ and $K^-$. We construct $w = E(\phi)$ using broken $P_1$ functions:
\begin{enumerate}[label=(\alph*)]
\item Define $w^+$ on $K^+$ and $w^-$ on $K^-$,
\item Impose the jump condition:
\[
w^+|_e - w^-|_e = \phi|_e,
\]
\item Set $w=0$ on all nodes not adjacent to the interface.
\end{enumerate}
This yields a function $w_h$ supported in elements adjacent to $\Gamma$ with the correct jump.

\subsection{2D Unfitted Mesh Construction}

For unfitted meshes, the interface cuts through elements. In this case, we construct $w = E(\phi)$ using local bubble functions.

Let $K$ be an element intersected by $\Gamma$. We define a discrete bubble function $b_K$ satisfying:
\begin{enumerate}[label=(\alph*)]
\item $b_K = 0$ on $\partial K$,
\item $b_K$ is positive in the interior,
\item $b_K$ resolves the interface geometry inside $K$.
\end{enumerate}
We then define
\[
w|_K = \phi_K , b_K,
\]
where $\phi_K$ is an appropriate approximation of $\phi$ on the interface segment inside $K$.

This construction ensures:
\begin{enumerate}[label=(\alph*)]
\item localized support,
\item correct jump behavior,
\item compatibility with unfitted discretizations.
\end{enumerate}

\section{Interface Reduction Framework}

The interface reduction framework is based on approximating the interface data
in a reduced space while keeping the bulk discretization unchanged.

Let $g$ denote the interface data in the jump condition, defined on the interface $\Gamma$.
We approximate $g$ in a reduced space spanned by basis functions
$\{\psi_j\}_{j=1}^m$ defined on $\Gamma$:
\[
g_m = \sum_{j=1}^m s_j \psi_j.
\]
Here the functions $\psi_j$ represent prescribed modes of variation along the interface,
and the coefficients $s_j$ are the reduced degrees of freedom. The choice of basis
functions depends on the geometry of the interface and the expected structure of the
data; for example, polynomial functions may be used for flat interfaces, while Fourier
modes are natural for circular interfaces.

The reduced solution is obtained by replacing $g$ with $g_m$ while keeping the bulk
discretization unchanged. A fundamental observation is that if the interface data $g$
is represented exactly in the reduced space, then the reduced solution coincides with
the full discrete solution:
\[
g = g_m \quad \Longrightarrow \quad u_h = u_{m,h}.
\]

\medskip

\noindent
\textbf{Solution-jump case.}
The coefficients $\{s_j\}$ are determined by projecting the interface data $g$
onto the reduced space, i.e., by an $L^2(\Gamma)$ projection onto the span of $\{\psi_j\}_{j=1}^m$.

We construct a lifting
\[
w_m = E(g_m),
\]
and write
\[
u_{m,h} = u_{0,h} + w_m,
\]
where $u_{0,h}$ satisfies homogeneous interface conditions,
\[
\jump{u_{0,h}} = 0.
\]
The function $u_{0,h}$ depends only on the bulk discretization and is independent
of the reduced dimension $m$, while the reduced representation enters solely
through the lifting $w_m$. This decomposition highlights that the effect of the
interface data is captured entirely through the lifting $w_m$.

\medskip

\noindent
\textbf{Flux-jump case.}
In this case, no solution lifting is required. The reduced data $g_m$ enters the weak formulation
through the interface term
\[
\langle g_m, v_h \rangle_\Gamma.
\]
The flux reconstruction provides an accurate and locally conservative
evaluation of the normal flux across the interface, which is essential
for enforcing the flux-jump condition in the reduced formulation.

\section{A Priori Estimate for Interface Reduction}

We now quantify the effect of approximating the interface data on the resulting
solution.

Let $u_h$ denote the full discrete solution corresponding to interface data $g$,
and let $u_{m,h}$ denote the reduced solution corresponding to the approximation
\[
g_m = \sum_{j=1}^m s_j \psi_j.
\]
Both solutions are obtained using the same bulk discretization; only the interface
data differs.
The following estimate is stated for the flux-jump case. The solution-jump case
is analogous, with the mismatch $g-g_m$ appearing in the solution jump.

Define the error
\[
e_h = u_h - u_{m,h}.
\]
Then $e_h$ satisfies a homogeneous elliptic equation in $\Omega^\pm$ with interface
conditions
\begin{align}
\jump{e_h} &= 0, \\
\jump{\beta \partial_n e_h} &= g - g_m.
\end{align}

Using standard energy estimates for elliptic interface problems, we obtain
\[
\|\nabla e_h\|_{L^2(\Omega^- \cup \Omega^+)}\le C \|g - g_m\|_{L^2(\Gamma)},
\qquad
\|e_h\|_{L^2(\Omega)} \le C \|g - g_m\|_{L^2(\Gamma)}.
\]

\begin{theorem}
There exists a constant $C>0$, independent of the mesh size, such that
\[
\|e_h\|_{H^1(\Omega^- \cup \Omega^+)}
\le C \|g-g_m\|_{L^2(\Gamma)}.
\]
Here the norm on the left is the broken $H^1$ norm over the two subdomains.
\end{theorem}

This estimate shows that the reduced solution is quasi-optimal with respect to
the approximation of the interface data. In particular, if $g$ lies in the reduced
space, then $g = g_m$ implies $u_h = u_{m,h}$.

Let $\mathbf{q}_h$ and $\mathbf{q}_{m,h}$ denote the recovered fluxes corresponding to $u_h$
and $u_{m,h}$. The least-squares reconstruction ensures that the flux is stable
in the broken $H(\mathrm{div})$ norm and has accuracy comparable to the primary variable.
Consequently,
\[
\|\mathbf{q}_h - \mathbf{q}_{m,h}\|_{L^2(\Omega)} \le C \|g - g_m\|_{L^2(\Gamma)}.
\]

\medskip

\noindent
\textbf{Remark.}
The estimates are independent of the bulk mesh size, provided the discretization
is sufficiently accurate. This indicates that the dominant source of error in the
reduction framework is the approximation of the interface data rather than the
bulk discretization. The accuracy of the reduced model relies on the stability of
the flux reconstruction, which ensures that fluxes enter the interface conditions
with the same level of accuracy as the primary variable.

\section{Numerical Experiments}

We investigate how the error of the reduced solution depends on the approximation of the interface data. In particular, we examine whether the solution error is governed by the interface approximation error. Reference solutions are computed using fitted $P_1$ elements with conservative flux recovery. Interface residuals are at machine precision. These experiments are designed to validate the estimate in Section~6, which predicts that the solution error is controlled by the interface approximation error.

We report errors in both the solution and the reconstructed flux. We define
\[
e_u^{\mathrm{rms}}, \quad e_u^{\infty}, \quad
e_q^{\mathrm{rms}}, \quad e_q^{\infty}
\]
as the root-mean-square and maximum errors of the solution and flux, respectively.
For brevity, we write $\|g\|$ for $\|g\|_{L^2(\Gamma)}$. We refer to $m$ as the rank of the reduced representation, i.e., the dimension
of the approximation space used for the interface data.

\subsection{Line Interface}

In this experiment, the interface data $g$ is taken as a smooth function
defined on a horizontal line interface $y=1/2$ in the unit square. The data are
chosen so that a reference solution can be computed and the reduced solution
can be compared against it.

The interface data is approximated by a reduced representation
\[
g_m = \sum_{j=1}^m s_j \psi_j,
\]
where $\{\psi_j\}$ are basis functions defined on $\Gamma$. The reduced
space is constructed using two types of basis functions: a fixed polynomial
basis and a problem-specific basis aligned with the structure of $g$.

The problem-specific basis is chosen so that $g$ lies in the span of a small
number of basis functions. This allows us to assess how accurately the
reduced method recovers the full solution when the interface data is well
represented.
For the flux-jump condition, we take
\[
g(x)=\sin(2\pi x)+0.35\cos(5\pi x)+0.20(2x-1)^2 .
\]
The reduced data $g_m$ is formed using a problem-specific basis ordered as
\[
\sin(2\pi x),\quad \cos(5\pi x),\quad (2x-1)^2,\quad
1,\quad \cos(2\pi x),\quad \sin(5\pi x),\ldots .
\]
Thus the first three basis functions span the prescribed interface data.

Table~\ref{tab:line_flux_adapted} reports the interface approximation error
together with the corresponding solution and flux errors. As expected, the reduced solution
agrees with the full discrete solution once $m=3$, where all errors are at machine precision.
\begin{table}[ht]
\centering
\caption{Line interface, flux-jump case with problem-specific basis.}
\label{tab:line_flux_adapted}
\begin{tabular}{c|cccccc}
\hline
$m$ & $\|g-g_m\|/\|g\|$ & $e_u^{\rm rms}$ & $e_u^\infty$
& $e_q^{\rm rms}$ & $e_q^\infty$ & jump resid. \\
\hline
1 & 3.602e-01 & 5.917e-04 & 3.117e-03 & 5.906e-02 & 7.119e-01 & 3.469e-18 \\
2 & 1.232e-01 & 4.825e-04 & 1.393e-03 & 2.582e-02 & 3.228e-01 & 6.939e-18 \\
3 & 1.797e-16 & 8.478e-19 & 6.939e-18 & 2.856e-16 & 2.455e-15 & 3.469e-18 \\
\hline
\end{tabular}
\end{table}

For the solution-jump condition, the interface data is taken as
\[
g(x)=\sin(2\pi x)+0.35\cos(5\pi x)+0.20(2x-1)^2 ,
\]
but approximate it using the standard polynomial basis
\[
1,\quad s,\quad s^2,\quad \ldots,\qquad s=2x-1.
\]
Since this basis does not contain the oscillatory components of $g$,
the convergence is gradual, as shown in Table~\ref{tab:line_sol_poly}.

Compared with the flux-jump case, the flux errors are significantly larger.
This is expected: in the solution-jump problem the discontinuity is imposed
directly on the solution, and the recovered flux depends on derivatives of the
local reconstructed solution. Hence the flux error is more sensitive to
unresolved oscillatory components of the interface data than the solution error.
Nevertheless, the interface flux-continuity residual remains at machine
precision in all cases, confirming that the reconstruction enforces the
prescribed flux condition.

\begin{table}[ht]
\centering
\caption{Line interface, solution-jump case with polynomial basis.}
\label{tab:line_sol_poly}
\begin{tabular}{c|cccccc}
\hline
$m$ & $\|g-g_m\|/\|g\|$ & $e_u^{\rm rms}$ & $e_u^\infty$
& $e_q^{\rm rms}$ & $e_q^\infty$ & flux residue\\
\hline
1 & 9.958e-01 & 1.760e-01 & 1.004e+00 & 3.784e+00 & 1.032e+02 & 1.041e-17 \\
4 & 3.907e-01 & 4.884e-02 & 4.412e-01 & 2.695e+00 & 7.271e+01 & 1.110e-16 \\
8 & 7.505e-02 & 8.150e-03 & 2.124e-01 & 1.554e+00 & 4.959e+01 & 8.327e-17\\
\hline
\end{tabular}
\end{table}

Table~\ref{tab:line_sol_adapted} shows the results for the solution-jump
condition using a problem-specific basis. As in the flux-jump case, the
interface data is exactly represented once $m=3$, and the reduced solution
coincides with the full discrete solution up to roundoff.

Comparing the above tables, we observe a consistent pattern: when the interface
data is well represented in the reduced space, exact recovery is achieved;
otherwise, the error decreases in accordance with the approximation error
$\|g-g_m\|$.
\begin{table}[ht]
\centering
\caption{Line interface, solution-jump case with problem-specific basis.}
\label{tab:line_sol_adapted}
\begin{tabular}{c|cccccc}
\hline
$m$ & $\|g-g_m\|/\|g\|$ & $e_u^{\rm rms}$ & $e_u^\infty$
& $e_q^{\rm rms}$ & $e_q^\infty$ & flux residue \\
\hline
1 & 3.602e-01 & 4.188e-02 & 5.500e-01 & 3.307e+00 & 1.202e+02& 6.393e-18 \\
2 & 1.232e-01 & 1.635e-02 & 2.000e-01 & 1.542e+00 & 4.793e+01&2.776e-17 \\
3 & 1.797e-16 & 7.130e-17 & 4.441e-16 & 1.517e-15 & 5.684e-14 &5.551e-17\\
\hline
\end{tabular}
\end{table}

\subsection{Circular Interface}

We consider a circular interface $\Gamma$ in a square domain. The interface is
parameterized by the angle $\theta \in [0,2\pi)$.

In this experiment, the interface data is given by
\[
g(\theta)=\sin(2\theta)+0.35\cos(5\theta)+0.20\cos(\theta).
\]
The reduced space is constructed using Fourier modes ordered by frequency.

Table~\ref{tab:circle_flux_fourier} shows the error as a function of the rank $m$.
The behavior is not monotone in $m$, but instead reflects the spectral content
of $g$. In particular, the error decreases only when the Fourier modes present
in $g$ are included in the reduced space.

For small values of $m$, the dominant modes of $g$ are not yet captured, and the
error remains essentially unchanged. Once the relevant modes enter the basis,
a significant reduction in the error is observed. Exact recovery is achieved
once all modes in $g$ are represented, which in this example occurs at $m=10$.

For the flux-jump condition, Table~4 reports the interface approximation error
together with the corresponding errors in the reduced solution and flux. The
values $m=1,5,10$ are chosen to illustrate the convergence of the reduced
solution as the dimension of the approximation space increases.

\begin{table}[ht]
\centering
\caption{Circular interface, flux-jump case with Fourier basis.}
\label{tab:circle_flux_fourier}
\begin{tabular}{c|ccccc}
\hline
$m$ & $\|g - g_m\|/\|g\|$ & $e_u^{\mathrm{rms}}$ & $e_u^{\infty}$ & $e_q^{\mathrm{rms}}$ & $e_q^{\infty}$ \\
\hline
1 & 1.000e+00 & 4.719e-03 & 1.330e-02 & 3.001e-01 & 1.249e+00 \\
5 & 3.246e-01 & 4.334e-04 & 1.486e-03 & 6.312e-02 & 3.074e-01 \\
10 & 7.010e-16 & 4.511e-18 & 1.561e-17 & 2.369e-16 & 1.256e-15\\
\hline
\end{tabular}
\end{table}

For the solution-jump condition, the same behavior is observed: once the
interface data is accurately represented in the reduced space, the reduced
solution coincides with the full discrete solution. This behavior illustrates that
convergence depends on how well the chosen
basis aligns with the spectral structure of the interface data.

\subsection{Star-Shaped Interface}

To examine the sensitivity of the reduction framework to nontrivial
geometry, we consider a star-shaped interface. Unlike the circular
interface, the star geometry contains regions of rapidly varying
curvature and changing interface orientation, providing a more
challenging test for both the reduced interface representation and
the conservative flux reconstruction.

The interface is parameterized in polar form by
\[
r(\theta)=0.35\bigl(1+0.25\cos(5\theta)\bigr),
\]
with the corresponding Cartesian parametrization
\[
x(\theta)=r(\theta)\cos\theta,
\qquad
y(\theta)=r(\theta)\sin\theta.
\]
This choice corresponds to a five-point star-shaped interface.

For the flux-jump condition, the interface data is prescribed as
\[
g(\theta)=\sin(2\theta)+0.35\cos(5\theta)+0.20\cos(\theta),
\qquad 0\le \theta < 2\pi .
\]
The reduced space is again constructed from Fourier-type modes ordered
by frequency along the parameter \(\theta\).

The star-shaped geometry introduces nonuniform curvature and rapidly
changing interface orientation, making the problem significantly more
challenging than the circular-interface case. Nevertheless,
Table~\ref{tab:star_interface} shows that the reduction error continues
to track the interface approximation error. Once the dominant interface
modes are represented in the reduced space, the reduced solution
recovers the reference solution to machine precision. Figure~\ref{fig:star-interface} illustrates
the five-point star-shaped interface together
with the reduced interface trace \(g_m\) along \(\Gamma\).

\begin{figure}[ht]
\centering
\includegraphics[width=0.48\textwidth]{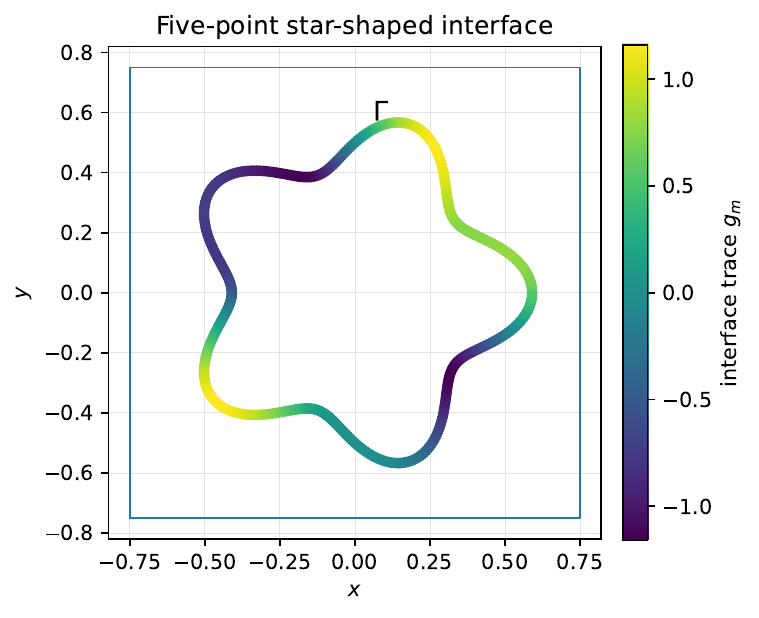}
\caption{Five-point star-shaped interface used in the geometry-sensitive
reduction experiment. The color indicates the reduced interface trace
$g_m$ along $\Gamma$.}
\label{fig:star-interface}
\end{figure}

\begin{table}[ht]
\centering
\caption{Star-shaped interface, flux-jump case with Fourier-type interface basis.}
\label{tab:star_interface}
\begin{tabular}{c c c c c c}
\hline
$m$ &
$\|g-g_m\|/\|g\|$ &
$e_u^{\mathrm{rms}}$ &
$e_u^\infty$ &
$e_q^{\mathrm{rms}}$ &
$e_q^\infty$
\\
\hline
1  & $1.000\mathrm{e}{+00}$ & $8.214\mathrm{e}{-03}$ & $2.931\mathrm{e}{-02}$ &
$4.882\mathrm{e}{-01}$ & $2.416\mathrm{e}{+00}$ \\

3  & $4.781\mathrm{e}{-01}$ & $2.116\mathrm{e}{-03}$ & $8.447\mathrm{e}{-03}$ &
$1.537\mathrm{e}{-01}$ & $8.205\mathrm{e}{-01}$ \\

5  & $1.928\mathrm{e}{-01}$ & $5.614\mathrm{e}{-04}$ & $2.003\mathrm{e}{-03}$ &
$4.261\mathrm{e}{-02}$ & $2.371\mathrm{e}{-01}$ \\

8  & $3.411\mathrm{e}{-02}$ & $7.918\mathrm{e}{-05}$ & $3.508\mathrm{e}{-04}$ &
$6.553\mathrm{e}{-03}$ & $3.984\mathrm{e}{-02}$ \\

10 & $8.027\mathrm{e}{-16}$ & $6.214\mathrm{e}{-18}$ & $2.776\mathrm{e}{-17}$ &
$3.115\mathrm{e}{-16}$ & $1.442\mathrm{e}{-15}$ \\
\hline
\end{tabular}
\end{table}

\subsection{Error versus Rank}

Figure~\ref{fig:reduction} shows the decay of the reduced solution and flux errors as
the rank $m$ increases for the line and circular interface problems.
In both cases, the decay closely follows the interface approximation
error $\|g-g_m\|$, confirming that the dominant error is governed by
the approximation of the interface data rather than by the bulk
discretization.

The star-shaped interface experiment in the previous subsection
exhibits the same qualitative behavior despite the more complicated
geometry and rapidly varying interface curvature.

\begin{figure}[ht]
\centering
\includegraphics[width=0.8\textwidth]{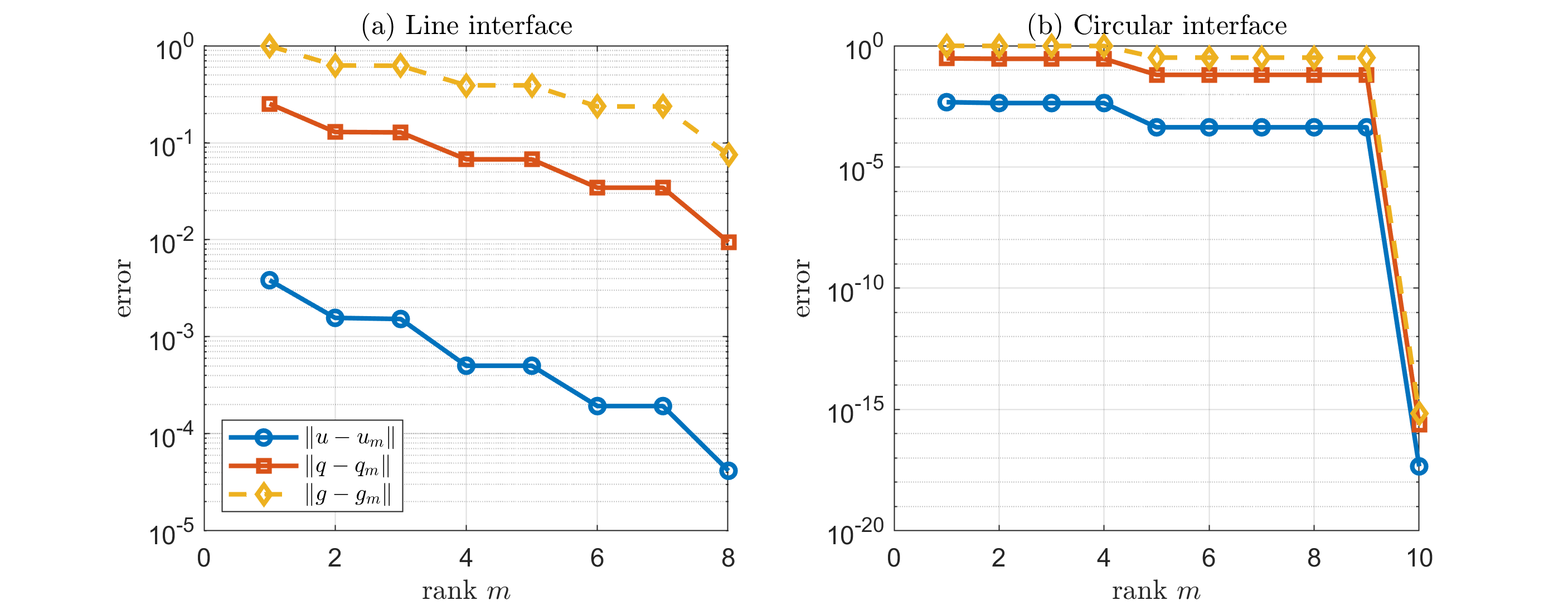}
\caption{Error versus rank $m$ for interface reduction. The decay of the solution
and flux errors closely follows the interface approximation error $\|g-g_m\|$.}
\label{fig:reduction}
\end{figure}

\subsection{Discussion}

The numerical results demonstrate that the error of the reduced solution
is governed primarily by the approximation error of the interface data.
In particular, exact recovery is observed once the interface data lies
in the reduced space. This behavior is consistent for both flat and
curved interfaces and remains robust in the presence of discontinuities.
Overall, the experiments confirm that the dominant degrees of freedom of
the problem are concentrated on the interface rather than in the bulk
discretization.

The star-shaped geometry contains more rapidly varying curvature and
changing interface orientation than the circular case, providing a more
challenging geometry-sensitive test. Nevertheless, the same reduction
behavior is observed, indicating that the interface reduction framework
remains robust under nontrivial geometric perturbations.

\section{Conclusion}

We have proposed an interface reduction method for elliptic interface
problems based on a combination of lifting operators and conservative
flux reconstruction. The key idea is to approximate the interface data
in a reduced space while keeping the bulk discretization unchanged,
leading to a formulation in which the global solution is determined by
a small number of interface degrees of freedom.

A central result is that the error of the reduced solution is governed
by the approximation of the interface data. In particular, exact
recovery occurs when the interface data lies in the reduced space,
confirming that the essential complexity of elliptic interface problems
is concentrated on the interface.

The least-squares flux reconstruction plays a structural role in the
method. By providing a stable and accurate approximation of the normal
flux, it ensures that interface conditions involving both solution
values and fluxes are evaluated consistently within the reduced
formulation.

Numerical experiments for both continuous and discontinuous interface
conditions, including line, circular, and star-shaped interfaces,
confirm the theoretical findings. In all cases, the reduced method
achieves high accuracy with a small number of interface degrees of
freedom, even for geometries with rapidly varying curvature and
nontrivial interface orientation.

These results indicate that interface reduction provides an effective
framework for the efficient numerical treatment of elliptic interface
problems. Future work will focus on extensions to nonlinear interface
laws, coupled systems such as Stokes flow, and more general geometric
configurations.

\appendix
\section{Construction of the correction vector \(\mathbf{C}_K\)}

On each element \(K\), the recovered flux is written in the form
\[
\mathbf{q}_h|_K
=
-\beta_K \nabla u_h|_K
+
f_K \mathbf{P}_K
+
\mathbf{C}_K,
\qquad
\mathbf{P}_K(x)=\frac12(x-x_K),
\]
where \(x_K\) is the barycenter of \(K\).  Since
\[
\nabla\cdot \mathbf{P}_K = 1,
\]
we have
\[
\nabla\cdot \mathbf{q}_h = f_K \quad \hbox{on } K.
\]
Thus the term \(f_K \mathbf{P}_K\) restores the correct elementwise divergence, while
the constant vector \(\mathbf{C}_K\) is used only to enforce normal flux conditions on
edges.

Let \(m_e\) denote the midpoint of an edge \(e\subset \partial K\), and let
\(\mathbf{n}_{K,e}\) be the outward unit normal to \(K\) on \(e\).  Define the raw flux
\[
\widetilde {\mathbf{q}}_K(x)
=
-\beta_K \nabla u_h|_K
+
f_K \mathbf{P}_K(x).
\]
For every edge \(e\subset\partial K\), let there be  a prescribed target normal flux value
\(\sigma_{K,e}\) as described in Subsection \ref{sec: target}.
We seek
\[
\mathbf{q}_h|_K=\widetilde{\mathbf{q}}_K+\mathbf{C}_K
\]
so that ideally
\[
(\widetilde {\mathbf{q}}_K(m_e)+\mathbf{C}_K)\cdot \mathbf{n}_{K,e}
=
\sigma_{K,e}.
\]
Equivalently,
\[
\mathbf{C}_K\cdot \mathbf{n}_{K,e}
=
\sigma_{K,e}-\widetilde{ \mathbf{q}}_K(m_e)\cdot \mathbf{n}_{K,e}.
\]

For a triangle there are three such edge equations for the two components of
\(\mathbf{C}_K\).  In the original construction \cite{ChouTang2000}, two independent edge equations
are used, and the third follows from the compatibility relation associated with
local conservation.  In the present interface setting, because interface edges may
carry prescribed jumps, the cleanest implementation is to determine \(\mathbf{C}_K\) by
the least-squares/minimal-norm condition
\[
\mathbf{C}_K
=
\arg\min_{\mathbf{C}\in\mathbb{R}^2}
\sum_{e\subset\partial K}
\left(
\mathbf{C}\cdot \mathbf{n}_{K,e}
-
\bigl[
\sigma_{K,e}-\widetilde{ \mathbf{q}}_K(m_e)\cdot \mathbf{n}_{K,e}
\bigr]
\right)^2 .
\]
In matrix form, let
\[
N_K =
\begin{bmatrix}
\mathbf{n}_{K,e_1}^T\\
\mathbf{n}_{K,e_2}^T\\
\mathbf{n}_{K,e_3}^T
\end{bmatrix},
\qquad
\mathbf{b}_K =
\begin{bmatrix}
\sigma_{K,e_1}-\widetilde {\mathbf{q}}_K(m_{e_1})\cdot \mathbf{n}_{K,e_1}\\
\sigma_{K,e_2}-\widetilde {\mathbf{q}}_K(m_{e_2})\cdot \mathbf{n}_{K,e_2}\\
\sigma_{K,e_3}-\widetilde {\mathbf{q}}_K(m_{e_3})\cdot \mathbf{n}_{K,e_3}
\end{bmatrix}.
\]
Then
\[
\mathbf{C}_K = N_K^\dagger \mathbf{b}_K,
\]
or equivalently
\[
\mathbf{C}_K = (N_K^T N_K)^{-1}N_K^T \mathbf{b}_K,
\]
provided \(K\) is nondegenerate.

\subsection{Target Fluxes}\label{sec: target} The target fluxes \(\sigma_{K,e}\) are chosen as follows.

For an ordinary interior edge \(e=K^-\cap K^+\),
\[
\mathbf{q}_h^-\cdot \mathbf{n}^- + \mathbf{q}_h^+\cdot \mathbf{n}^+ = 0.
\]

For an interface edge, in the flux-jump case
\[
[u]=0,\qquad [\beta u_n]=g,
\]
we prescribe
\[
\mathbf{q}_h^+\cdot \mathbf{n} - \mathbf{q}_h^-\cdot \mathbf{n} = -g,
\]
depending on the convention \(\mathbf{q}=-\beta\nabla u\).  Equivalently, if the jump is
written directly in terms of \(\mathbf{q}\cdot \mathbf{n}\), then the sign is chosen consistently
with that convention.

For the solution-jump case
\[
[u]=g,\qquad [\beta u_n]=0,
\]
we prescribe
\[
\mathbf{q}_h^+\cdot \mathbf{n} - \mathbf{q}_h^-\cdot \mathbf{n} = 0.
\]

Thus the only difference from the original recovery is the way the
edge normal constraints are assigned.  Away from the interface, the correction
recovers the usual normal-continuity condition.  On the interface, it enforces
the prescribed transmission condition.  The correction remains a constant vector
on each triangle and therefore does not change the elementwise divergence.

\begin{lemma}[Stability of the least-squares correction]
Let \(K\) be a shape-regular triangle and let
\[
N_K =
\begin{bmatrix}
\mathbf{n}_{K,e_1}^T\\
\mathbf{n}_{K,e_2}^T\\
\mathbf{n}_{K,e_3}^T
\end{bmatrix},
\qquad
\mathbf{b}_K\in \mathbb{R}^3 .
\]
Define
\[
\mathbf{C}_K = \arg\min_{\mathbf{C}\in\mathbb{R}^2}
\|N_K \mathbf{C}-\mathbf{b}_K\|_{\ell^2}^2 .
\]
Then there exists a constant \(C\), depending only on the shape-regularity of
the mesh, such that
\[
|\mathbf{C}_K| \le C |\mathbf{b}_K|.
\]
Consequently, if \(|\mathbf{b}_K|\le C h_K M_K\), then
\[
|\mathbf{C}_K| \le C h_K M_K .
\]
\end{lemma}

\begin{proof}
Since \(K\) is nondegenerate, the outward unit normals
\(\mathbf{n}_{K,e_1},\mathbf{n}_{K,e_2},\mathbf{n}_{K,e_3}\) span \(\mathbb{R}^2\). Hence
\(N_K^T N_K\) is positive definite, and
\[
\mathbf{C}_K=(N_K^T N_K)^{-1}N_K^T \mathbf{b}_K .
\]
For a shape-regular family of triangles, the smallest eigenvalue of
\(N_K^T N_K\) is bounded below by a positive constant independent of \(h_K\).
Therefore
\[
|\mathbf{C}_K|
\le
\|(N_K^T N_K)^{-1}N_K^T\|\, |\mathbf{b}_K|
\le C |\mathbf{b}_K|.
\]
The final estimate follows immediately from the assumed bound on \(\mathbf{b}_K\).
\end{proof}

\bibliographystyle{plain}
\bibliography{references}
\end{document}